\documentclass[12pt, a4paper]{amsart}
\usepackage[a4paper, textheight=24truecm, textwidth=16truecm]{geometry}
\usepackage{amsmath, amssymb, amsthm, amsfonts, amscd}
\usepackage{hyperref}
\usepackage[alphabetic]{amsrefs}
\theoremstyle{plain}
\newtheorem{theorem}{Theorem}[subsection]
\newtheorem{proposition}[theorem]{Proposition}
\newtheorem{lemma}[theorem]{Lemma}

\theoremstyle{remark}
\newtheorem{rema}[theorem]{Remark}

\theoremstyle{definition}
\newtheorem{definition}[theorem]{Definition}
\newcommand\ob{^{-1}}
\newcommand\nf{\normalfont}
\newcommand\zop{{\mathbb{Z}[\frac{1}{p}]}}
\DeclareMathOperator{\obj}{Obj}

\DeclareMathOperator{\id}{id}
\newcommand\cu{\underline{C}}
\newcommand\du{{\underline{D}}}
\newcommand\eu{{\underline{E}}}
\newcommand\hw{{\underline{Hw}}}
\newcommand\Z{{\mathbb{Z}}}
\newcommand\Q{{\mathbb{Q}}}
\newcommand\p{\mathbb{P}}
\newcommand\ns{\{0\}}
\DeclareMathOperator\chow{{Chow}}
\DeclareMathOperator\spe{Spec}
\DeclareMathOperator\imm{{Im}}

\newcommand\dm{\mathcal{DM}}
\newcommand\dmc{\dm_c}
\newcommand\dms{\dm(S)}
\newcommand\dmcs{\dm_c(S)}
\newcommand\dmx{\dm(X)}
\newcommand\dmcx{\dm_c(X)}
\newcommand\dmck{\dm_c(K)}

\newcommand\dmy{\dm(Y)}

\newcommand\dmgm{\dm_{gm}}
\newcommand\chows{\chow(S)}
\newcommand\hwchow{{\protect\underline{Hw}}_{\chow}}

\newcommand\hwchows{{\hwchow(S)}}

\newcommand\wchowb{{w_{\chow}^{big}}}

\newcommand\op{\mathcal{OP}}
\newcommand\on{\mathcal{ON}}
\newcommand\oz{\mathcal{OZ}}
\newcommand\opx{\op(X)}
\newcommand\onx{\on(X)}
\newcommand\ozx{\oz(X)}
\newcommand\opal{\op(\alpha)}
\newcommand\onal{\on(\alpha)}

\newcommand\mgbms{\mathcal{M}^{BM}_S}
\newcommand\mgbmy{\mathcal{M}^{BM}_Y}
\newcommand\mgbm{\mathcal{M}^{BM}}
\newcommand\mgbmx{\mathcal{M}^{BM}_X}
\newcommand\z{{\mathbb{Z}}}
\numberwithin{equation}{section}
\begin{document}
\title[Chow weight  structures for cdh-motives 
]{On Chow weight  structures for $cdh$-motives with integral coefficients
}
\author{Mikhail V. Bondarko, Mikhail A. Ivanov}
\thanks{The work is supported by RFBR (grants no. 14-01-00393A and 15-01-03034A). The first author is also grateful to the Dmitry Zimin's Foundation "Dynasty".
} 
\begin{abstract} 
The main goal of this paper is to define a certain {\it Chow weight
structure} $w_{\chow}$ on the category
$\dmcs$ of (constructible) $cdh$-motives over an equicharacteristic
scheme $S$. 
In contrast to the
previous papers of D. H\'ebert and the first author on weights for
relative motives (with rational coefficients), 
we can 
 achieve our goal for motives
with integral coefficients (if $\operatorname{char} S=0$; if $\operatorname{char} S=p>0$ then we
consider motives with $\Z[\frac{1}{p}]$-coefficients). We prove that the
properties of the Chow weight structures that were previously
established for $\Q$-linear motives can be carried over to this "integral"
context (and we generalize some of them using certain new methods). In this paper we mostly study
the version of $w_{\chow}$ defined via "gluing from strata"; this enables
us to define  Chow weight structures 
for a wide class of base schemes.

As a consequence, we certainly obtain certain (Chow)-weight spectral
sequences and filtrations on any (co)homology of motives.

\end{abstract}

\maketitle 
\tableofcontents
 \section*{Introduction}
In this paper we construct certain "weights" for $R$-linear motives over a scheme $S$.
  Here $S$ is an excellent finite dimensional Noetherian
scheme of characteristic
$p$ (that could be $0$) and $R$ is a unital commutative associative
coefficient ring; in the case  $p>0$ we require $p$ to be invertible
in $R$. These weights are compatible with Deligne's weights for 
constructible complexes of \'etale sheaves (see Remark \ref{rintel}(4) below).

Now we explain this in more detail.
 Deligne's weights for \'etale sheaves and mixed Hodge structures (and for the corresponding derived categories) are
very important for modern algebraic geometry.
So, lifting these weights to motives is an important part of  the so-called
Beilinson's (motivic) dream.
 The 'classical' approach (due to Beilinson) to 
do this is to define a filtration on motives that would split Chow
motives into their components corresponding to single (co)homology
groups (i.e., it should yield the so-called Chow-Kunneth
decompositions). Since  the existence of Chow-Kunneth decompositions
is very much conjectural,
 it is no wonder  that this approach 
was not really successful (up to now); besides, it cannot work for $R$-linear motives if $R$ is not a $\Q$-algebra.

  In \cite{bws} an alternative method for defining  weights for motives
was proposed and successfully implemented: the so-called Chow weight
structure on the triangulated category of Voevodsky motives $\dmgm$ (with integral coefficients) over
a characteristic $0$ field was defined; the {\it heart} of this weight structure is the "classical" category of Chow motives. Now, arbitrary weight
structures yield functorial weight
filtrations and weight spectral sequences for any (co)homological
functor (from $\dmgm$). These weight filtrations and spectral sequences generalize
Deligne's ones; note that they are also well-defined for any
(co)homology with integral coefficients (this is a vast extension of
the earlier results of \cite{gs} on cohomology with compact support)!

The next paper in this direction was \cite{bmres}, where the Chow weight
structure for $\zop$-linear motives over a perfect field of
characteristic $p$ was defined. At the same time, the theory of
Voevodsky triangulated motivic categories over any "more or less general" base scheme
$S$ was (in \cite{degcis}) developed to the stage that the Chow weight
structure for Beilinson motives over $S$ (i.e., for $S$-motives with
rational coefficients) could be defined; the latter was done
independently in \cite{hebpo} and in \cite{brelmot} (see Remark \ref{rchow}(2) below).

The main goal of the current paper is to define the Chow weight
structure on $\zop$-linear motives (and more generally, on
$R$-linear motives for any $\zop$-algebra $R$) over any excellent  finite dimensional Noetherian base
scheme $S$
of characteristic $p$ (here we set $\zop=\Z$ if $p=0$, and consider
$cdh$-motives with $R$-coefficients that were denoted by $\dm_{cdh}(S,R)$ in \cite{cdintmot}).
To achieve this we 
 use the "gluing
construction" of
the Chow weight structure; this construction  was described (for $\Q$-linear motives)
in \cite{brelmot}*{\S2.3} (whereas the method was first proposed in \cite{bws}*{\S8.2}). This required us to develop some new methods for
studying morphisms between relative motives (in \S1.3).
We also note that all the properties and
applications of the Chow weight structure described in \cite{brelmot} carry
over to our "integral" context. We 
 apply some new arguments for 
 studying the weight-exactness of motivic functors (in \S\ref{sfwchow}; following \cite{brhtp}, we use {\it Borel-Moore motives});
they allow us to extend the corresponding results to 
 not necessarily quasi-projective morphisms. 

Thus this paper gives  convenient tools for
studying "integral" (and torsion) weight phenomena for (equicharacteristic) schemes and motives. 
In particular, we obtain functorial "Chow-weight" filtrations and spectral sequences (see \S\ref{schws}). 
Note still that we are able to prove that "explicit Chow motives" over $S$
yield a weight structure for $S$-motives (similarly to \cite{hebpo}
and \cite{brelmot}*{ \S2.1}) only if $S$ is a "smooth limit" of  schemes of finite
type over some field (see \S\ref{stmain} and \cite{jin}).

Lastly, we note that one can (certainly) consider motivic categories corresponding to Grothendieck topologies distinct from the $cdh$ one.
In particular, a (not really "successful")  attempt was made
in \cite{bondint} to construct certain Chow weight structures on relative Nisnevich motivic categories (using their properties established in \cite{degcis}). Note yet that the Nisnevich motives are isomorphic to the $cdh$-ones over regular bases (see \cite{cdintmot}*{Corollary 5.9}); this is also expected to be true in general. On the other hand, though ($R$-linear) \'etale motivic categories (that were thoroughfully studied in \cite{cdet}) enjoy several "nice" properties,  there is no chance to define $w_{\chow}$ for 
them unless $\Q\subset R$, whereas in the latter case the relative motivic categories mentioned 
"do not depend on the choice of a topology" (if we compare $cdh$, Nisnevich, \'etale, and $h$-motives).

The authors are deeply grateful to prof. F. Deglise for his very helpful explanations. The first author 
expresses his gratitude to  Unit\'e de math\'ematiques pures et appliqu\'ees of
\'Ecole normale sup\'erieure de Lyon for the wonderful working conditions in January of 2015.

\section{Preliminaries}\label{sprem}
This section is mostly a  recollection of   basics on (relative $cdh$-)motives and weight structures; yet the results of \S\ref{sortl}
and the methods of their proofs are (more or less) new.

\subsection{Notation}\label{snotata}
\begin{itemize}
\item For categories $C$, $D$ we write $D\subset C$ if $D$ is a full subcategory of $C$.

 \item For a category $C$, and $X,Y\in\obj C$, we denote by $C(X,Y)$ the set of  $C$-morphisms from  $X$ to $Y$.

 \item An additive subcategory $D$ of $C$ is 
 said to be \emph{Karoubi-closed} in it if it contains all retracts  of its objects in $C$. 
The full subcategory of $C$ whose objects are all retracts of objects of $D$ (in $C$) will be called the \emph{Karoubi-closure} of $D$ in $C$.

\item $\cu$ will always denote some triangulated category; usually it will be endowed with a weight structure $w$ (see Definition \ref{dwstr} below).

\item For a set of objects $C_i\in\obj\cu$, $i\in I$, we will denote by $\langle C_i\rangle$ the smallest strictly full triangulated subcategory of $\underline{C}$ containing all $C_i$; for
$D\subset \cu$ we will write $\langle D\rangle$ instead of $\langle \obj D\rangle$. 
We will call the  Karoubi-closure of $\langle C_i\rangle$ in $\cu$ the  \emph{ triangulated category  generated by $C_i$} (recall that it is indeed triangulated).

\item For $X,Y\in \obj \cu$ we will write $X\perp Y$ if $\cu(X,Y)=\ns$.
For $D,E\subset \obj \cu$ we will write $D\perp E$ if $X\perp Y$ for all $X\in D$, $Y\in E$.
For $D\subset \cu$ we will denote by $D^\perp$ the class 
$$
\{Y\in \obj \cu\ :\ X\perp Y\ \forall X\in D\}.
$$
Dually, ${}^\perp{}D$ is the class $\{Y\in \obj \cu\ :\ Y\perp X\ \forall X\in D\}$.

\item We will say that some $C_i\in\obj\cu$, $i\in I$, \emph{weakly generate} $\cu$ if for $X\in\obj\cu$ we have: $\cu(C_i[j],X)=\ns\ \forall i\in I,\ j\in\Z\implies X=0$ (i.e., if $\{C_i[j]:\ j\in \Z\}^\perp$ contains only zero objects).

\item $M\in \obj \cu$ will be called compact if the functor $\cu(M,-)$
commutes with all small coproducts that exist in $\cu$ 
(we will only consider compact objects in those categories that are closed with respect to arbitrary small coproducts).

\item $D\subset \obj \cu$  will be called \emph{extension-stable} if $0\in D$ and for any distinguished triangle $A\to B\to C$ in $\cu$ we have: $A,C\in D\implies B\in D$.

\item We will call the smallest Karoubi-closed extension-stable subclass of $\obj\cu$ containing $D$ the \emph{envelope} of $D$.

\item We will sometimes need certain stratifications of a scheme $S$. Recall that  a  stratification  $\alpha$  is a presentation of  $S$ as $\cup S_\ell^\alpha$, where $S_\ell^\alpha$, $1\le \ell\le n$, are pairwise disjoint locally closed subschemes of $S$. Omitting $\alpha$, we will denote by $j_\ell:S_\ell^\alpha\to S$  the corresponding immersions.
 We do not demand the closure of each $S_\ell^\alpha$ to be  the union of strata (though we could do this); we will only assume that each $S_\ell^{\alpha}$ is open in $\cup_{i\ge \ell} S_i^{\alpha}$. 
 
\item Below we will identify a Zariski point (of a scheme $S$) with the spectrum of its residue field.

\item $k$ is a prime field, $p=\operatorname{char} k$ ($p$ may be $0$). 
\item 
All the schemes we consider will be excellent, separated, Noetherian  $k$-schemes (i.e., characteristic $p$ schemes) of finite Krull dimension (so, a "scheme" will always mean a scheme of this sort).

\item A variety over a field $F/k$ is a (separated) reduced scheme of finite type over $\spe F$. 

\item $S_{red}$ will denote the reduced scheme associated with $S$.

\item All morphisms of  schemes considered below will be separated.
They will also mostly be  of finite type. 

\item Throughout the paper $R$ will be some fixed  unital associative commutative  algebra over $\Z[\frac1p]$ (we set $\Z[\frac1p]=\Z$ if $p=0$).
\end{itemize}

\subsection{On $cdh$-motives (after Cisinski and Deglise)}\label{sbrmot}

We list some of the  properties of the triangulated categories of $
cdh$-motives (those are  certain relative Voevodsky motives 
with $R$-coefficients described by Cisinski and Deglise).  
They are very much similar to the properties of Beilinson motives (i.e., of $\Q$-linear ones) that were established in \cite{degcis} (and applied in \cite{hebpo} and \cite{brelmot} for the construction of the corresponding Chow weight structures).

\begin{theorem}\label{pcisdeg}
Let $X$, $Y$ be  {\nf(}Noetherian finite dimensional excellent characteristic $p$\/{\nf)} schemes{\nf\/;} $f:X\to Y$ is a  {\nf(}separated{\nf\/)}  morphism of finite type. 
\begin{enumerate}
\item\label{imotcat}
For any  $X$ there exists a tensor triangulated $R$-linear category $\dmx$  with the unit object $R_X$ {\nf(}in \cite{cdintmot}*{Definition 1.5} this category  was denoted by $\operatorname{DM}_{cdh}(X,R))${\nf\/;} it is closed with respect to arbitrary small coproducts.

\item\label{imotgen}
The {\nf(}full\/{\nf)} subcategory $\dmcx\subset \dmx$ of compact objects is tensor triangulated, and  $R_X\in \obj \dmcs$. $\dmcx$ weakly generates $\dmx$.

\item\label{imotfun} 
For any  $f$ the following functors are defined{\nf\/:} $f^*: \dm(Y) \leftrightarrows \dmx:f_*$ and $f_!: \dmx \leftrightarrows \dmy:f^!${\nf;} 
$f^*$ is left adjoint to $f_*$ and $f_!$ is left adjoint to $f^!$.

\noindent We call these {\bf the motivic image functors}. Any of them {\nf(}when $f$ varies\/{\nf)} yields a  $2$-functor from the category of {\nf(}separated finite-dimensional excellent characteristic $p$\/{\nf)} schemes with  morphisms of finite type to the $2$-category of triangulated categories. Besides, all motivic image functors preserve compact objects {\nf(}i.e., they can be restricted to the subcategories $\dmc(-));$ they also commute with arbitrary {\nf(}small\/{\nf)} coproducts. 

\item\label{iexch} 
For a Cartesian square of  morphisms of finite type 
$$\begin{CD}
X'@>{f'}>>Y'\\
@VV{g'}V@VV{g}V \\
X@>{f}>>Y
\end{CD}$$
we have $g^*f_!\cong f'_!g'{}^*$ and $g'_*f'{}^!\cong f^!g_*$.

\item\label{itate}
For any $X$ there exists a Tate object $R(1)\in\obj\dmcx;$ tensoring by it yields an  exact Tate twist functor $-(1)$ on $\dmx$.
This functor is an auto-equivalence of $\dmx;$ we will denote the inverse functor by $-(-1)$.
Tate twists commute with all motivic image functors mentioned {\nf(}up to an isomorphism of functors\/{\nf)}.
Besides, for $X=\p^1(Y)$ there is a functorial isomorphism $f_!(R_{\p^1(Y)})\cong R_Y\bigoplus R_Y(-1)[-2]$.

\item\label{iupstar}  $f^*$ is symmetric monoidal{\nf\/;} $f^*(R_Y)=R_X$.

\item \label{ipur}
$f_*\cong f_!$ if $f$ is proper{\nf\/;} $f^!(-)\cong f^*(-)(s)[2s]$  if $f$ is smooth  {\nf(}everywhere\/{\nf)} of relative dimension $s$. 
If $f$ is an open immersion, we just have $f^!=f^*$.

\item \label{ipura} 
If $i:S'\to S$ is an immersion of regular schemes everywhere of codimension $d$, then $R_{S'}(-d)[-2d]\cong i^!(R_S)$.

\item\label{iglu}
If $i:Z\to X$ is a closed immersion, $U=X\setminus Z$, $j:U\to X$ is the complementary open immersion, then the motivic image functors yield a {\it gluing datum} for $\dm(-)$  {\nf(}in the sense of \cite{bbd}*{\S1.4.3}{\nf;} see also \cite{bws}*{Definition 8.2.1}{\nf)}. That means that {\nf(}in addition to the adjunctions given by assertion {\nf\ref{imotfun})} the following statements are valid.
\begin{enumerate}
\item  $i_*\cong i_!$ is a full embeddings{\nf;} $j^*=j^!$ is isomorphic to the localization {\nf(}functor{\nf\/)} of $\dmx$ by $i_*(\dm(Z))$.

\item For any $M\in \obj \dmx$ the pairs of morphisms $j_!j^!(M) \to M \to i_*i^*(M)$ and $i_!i^!(M) \to M \to j_*j^*(M)$ can be uniquely completed to distinguished triangles {\nf(}here the connecting morphisms come from the adjunctions of assertion {\nf\ref{imotfun})}.

\item $i^*j_!=0${\nf;} $i^!j_*=0$.

\item All of the adjunction transformations $i^*i_*\to 1_{\dm(Z)}\to i^!i_!$ and $j^*j_*\to 1_{\dm(U)}\to j^!j_!$ are isomorphisms of functors.
\end{enumerate}

\item\label{igluc} 
For the subcategories $\dmc(-)\subset \dm(-)$ the obvious analogue of the previous assertion is fulfilled.

\item \label{itr}
If $f$ is a finite universal homeomorphism, then $f^*$, $f_*$, $f^!$,  and $f_!$  are equivalences of categories. 
Moreover, $ f^!R_Y\cong f^*R_Y =R_X$ and $f_*R_X= f_!R_X\cong R_Y$.

\item\label{igenc}
If $S$ is of finite type over a field, $\dmcs$ {\nf(}as a triangulated category\/{\nf)} is generated by $\{ g_*(R_X)(r)\}$, where $g:X\to S$ runs through all projective morphisms  
 such that $X$ is regular, $r\in \Z$.  

\item\label{icont}
Let a scheme $S$ be  the limit of an essentially affine {\nf(}filtering\/{\nf)} projective system of  schemes $S_\beta$ {\nf(}for $\beta\in B${\nf)}. 
Then $\dmcs$ is isomorphic to the $2$-colimit  of the categories $\dmc(S_\beta);$ in this isomorphism all the connecting functors are given by the corresponding motivic inverse image functors {\nf(}cf.  Remark \ref{ridmot}(1) below{\nf\/)}. 

\item\label{ibormo}  If $S$ is 
smooth over $k$ {\nf(}or over any other perfect field{\nf\/)},  
then for $b, c, r\in \Z$ with $r\ge 0$ we have $R_S(b)(2b)\perp R_S(c)[2c+r]$.
\end{enumerate}
\end{theorem}

\begin{proof}
These statements can (mostly) be found \cite{cdintmot}. Concretely:

(\ref{imotcat}) \cite{cdintmot}*{\S1.6}.

(\ref{imotgen}) Immediate from the definition of $\dm(-)$ given in \cite{cdintmot}*{\S1.5}. 

(\ref{imotfun})\cite{cdintmot}*{\S1.6, Theorem 6.4}. 

(\ref{iexch},\ref{itate},\ref{iupstar},\ref{ipur},\ref{iglu})  \cite{cdintmot}*{Proposition 4.3, Theorem 5.1} yield that $\dm(-)$ is a {\it motivic triangulated category}\/; 
see \cite{degcis}*{2.4.45} for the definition and 
\cite{degcis}*{2.4.50} for a list of properties that includes the assertions desired.



(\ref{ipura}) \cite{cdintmot}*{Proposition 6.2}.

(\ref{igluc}) This is an easy consequence of assertions \ref{imotfun} and \ref{iglu}; cf. (the proof of) Proposition 1.1.2(11) in \cite{brelmot}.

(\ref{itr}) By \cite{cdintmot}*{Proposition 7.1} the functors $f^*\cong f^!$ are  equivalences of categories. Hence their (right and left) adjoints $f_*$ and $f_!$ are equivalences also. 

(\ref{igenc}) \cite{cdintmot}*{Proposition 7.2}.

(\ref{icont}) \cite{cdintmot}*{Theorem 5.11}.

(\ref{ibormo}) \cite{cdintmot}*{Corollary 8.6, putting $X=\spe k$}. 
\end{proof}

\begin{rema}
\label{ridmot}
1. In \cite{cdintmot} the functor $g^*$ was constructed for any morphism $g:Y'\to Y$ not necessarily of finite type; it  preserves compact objects (see 
\S6.1(ii) of ibid.) and unit objects
$R_{-}$ (i.e., $g^*(R_{Y})=R_{Y'}$).  Besides, for any such $g$ and any  morphism $f:X\to Y$ of finite type we have an isomorphism $g^*f_!\cong f'_!g'{}^*$ (for the corresponding $f'$ and $g'$; cf. part \ref{iexch} of  our theorem).

We also note: if $f$ is a pro-open limit of immersions, then one can define $f^!=f^*$ (in particular, one can define $j_K^!$ for the natural morphism $j_K:K\to S$ if $K$ is a  Zariski point  of a scheme $S$; cf. \cite{bbd}*{\S2.2.12}); $f^!$ 
 preserves compact objects. 
For the functors of this type 
we also have $g'_*f'{}^!\cong f^!g_*$ (for  $g$ of finite type; see part \ref{iexch} of our theorem once more); see  \cite{degcis}*{Proposition 4.3.14}. 

2. For any
  morphism $f:X\to Y$ of finite type we set $\mgbm_Y(X)=f_!R_X$ (this is a certain {\it Borel-Moore} motif of $X$; cf. \cite{brhtp}, \cite{degchz}, and \cite{lemm}*{\S I.IV.2.4}). Note that Theorem \ref{pcisdeg}(\ref{itr}) easily yields that  $\mgbm_Y(X)\cong  \mgbm_Y(X_{red})$ (recall that $X_{red}$ is the reduced scheme associated with $X$). Besides, for any (separated) morphism $g:Y'\to Y$ the previous part of this remark yields that $g^*(\mgbm_Y(X))\cong \mgbm_{Y'}(X\times_Y Y')$.
\end{rema}

\begin{lemma}\label{filtr}
Let $S=\cup S_\ell^\alpha$ be a stratification. 
Then the following statements are valid.

\begin{enumerate}
\item  
For any $M,N\in \obj \dms$ there exists a filtration of $\dms(M,N)$  whose factors are certain subquotients of $\dm(S_\ell^\alpha)(j_\ell^*(M), j_\ell^!(N))$.
 If $M=R_S(a)[2a]$, $N=R_S(b)[2b+r]$ for $a,b,r\in \z$, $S$ is regular, and all $S_\ell^\alpha$ are regular and connected, then 
the factors of this filtration on $\dms(M,N)$ are  certain subquotients of $\dm(S_\ell^\alpha)(R_{S_\ell^\alpha}, R_{S_\ell^\alpha}(b-a-c_\ell)[2b+r-2a-2c_\ell])$, where $c_\ell$ is the codimension of $S_\ell$ in $S$.

\item 
If $g:S\to Y$ is a morphism of finite type then $\mgbmy(S)$ belongs to the 
 envelope of  $\mgbmy(S_{\ell,red}^{\alpha})$. 

\item Let $Z\subset X$ be a closed subscheme, where $X$ is a scheme of finite type over S{\nf\/;} denote by $U\subset X$ the complementary open subscheme. 
If $Z$ and $X$ are regular, $Z$ is everywhere of codimension $c$ in $X$, then there is a distinguished triangle
\begin{equation}\label{emgeffc}
z_* R_Z(-c)[-2c]\to x_* R_X\to u_*R_U
\end{equation} 
in $\dms$, where $z,x,u$ are the corresponding structure morphisms {\nf(}to $S)$.

\end{enumerate}\end{lemma}
\begin{proof}
1. The second part of the assertion is an simple consequence of the first one (see Theorem \ref{pcisdeg}(\ref{iupstar},\ref{ipura})).

We prove the first 
 statement by induction on the number of strata.
By definition (see  \S\ref{snotata})  $S_1^\alpha$ is open in $S$, and the remaining $S_\ell^\alpha$  yield a stratification of $S\setminus S_1^\alpha$.
We denote $S\setminus S_1^\alpha$ by $Z$, the (open) immersion $S_1^\alpha \to S$ by $j$ and the (closed) immersion $Z\to S$ by $i$.

Now, Theorem \ref{pcisdeg}(\ref{iglu}) yields a distinguished triangle 
\begin{equation}\label{eglu} j_!j^!(M) \to M\to i_*i^*(M).\end{equation} Hence there exists
 a (long) exact sequence
$$\dots\to \dms (i_*i^*(M),N)\to \dms(M,N)\to \dms (j_!j^!(M),N) \to\dots$$ 
The corresponding adjunctions of functors yield
$\dms(i_*i^*(M),N)\cong \dm(Z)(i^*(M), i^!(N))$ and $\dms(j_!j^!(M),N')\cong  \dm({S_1^\alpha})(j^*(M), j^!(N))$.

Now, by the inductive assumption the group $\dm(Z)(i^*(M), i^!(N))$ has a filtration whose factors are certain subquotients of $\dm(S_\ell^\alpha)(j_\ell^*(M), j^!(N))$ (for $\ell\neq 1$). This concludes the proof.

2.  We use the same induction and notation as in the previous proof. Considering the distinguished triangle (\ref{eglu}) for $M=R_S$ and applying $g_!$ to it we obtain a distinguished triangle \begin{equation}\label{emgys}
\mgbmy(S_1^\alpha) 
{\to}\mgbmy(S) 
{\to} \mgbmy(Z). 
\end{equation} 
In order to complete the inductive step it suffices to 
apply (the first statement in) 
Remark \ref{ridmot}(2).

3. We can assume that $X$ is connected. Theorem \ref{pcisdeg}(\ref{iglu}) yields a distinguished triangle $i_!i^!R_X \to 
R_X\to  j_*j^*R_X(\cong  j_*R_U) $ (see also part \ref{iupstar} of the theorem).
If $Z$ and $X$ are regular, then $i_!i^!R_X\cong i_*R_Z( -c)[-2c] $ (see part \ref{ipura} of 
the theorem). Hence the  application of $x_*$ to this distinguished triangle yields (\ref{emgeffc}).
\end{proof}

\subsection{Some orthogonality lemmas}\label{sortl}

The following motivic statements are very important for the current paper.

\begin{lemma}\label{ort1}
If $S$ is a regular scheme, then for any $a, b, r\in \Z$ with $r>0$ we have $R_S(a)[2a]\perp R_S(b)[2b+r]$.
\end{lemma}
\begin{proof}
We stratify $S$ as $\cup_{1\le \ell \le n} S_\ell$ 
 so that all $S_\ell$ are regular, affine  and connected.
For such a stratification (by Lemma \ref{filtr}(1)) we should prove  
$$\dm(S_\ell)(R_{S_\ell}(a-\operatorname{codim}_S S_\ell)[2a-2\operatorname{codim}_S S_\ell], R_{S_\ell}(b-\operatorname{codim}_S S_\ell)[2b-2\operatorname{codim}_S S_\ell+r])=\ns.$$
Thus it is sufficient to prove the statement for strata. 
That is, we can assume that $S$ is regular and affine.
Such a scheme $S$ can be presented as the inverse limit of regular schemes of finite type over $k$
(by the Popescu--Spivakovsky theorem; see \cite{degcis}*{Theorem 4.1.5}).
If $S=\varprojlim S_\beta$, then  $\dm_c(S)(R_S(b)[2b], R_S(c)[2c+r])=\varinjlim\dm_c(S_\beta)(R_{S_\beta}(b)[2b], R_{S_\beta}(c)[2c+r])$ by   Theorem \ref{pcisdeg}(\ref{icont}). In conclusion, we refer to part (\ref{ibormo}) of  this theorem. 
\end{proof}

\begin{rema}
\label{rmotcoh}
This continuity argument along with \cite{cdintmot}*{Corollary 8.6} also easily yields that  $\dm(R_S, R_S(b)[2b+r])$ is isomorphic to the corresponding higher Chow group of $S$, i.e., it can be computed using the Bloch or Suslin complex (of codimension $b$ cycles in $S\times \Delta^{-*}$) with $R$-coefficients, if $S$ is a regular affine scheme. This result cannot be automatically generalized to arbitrary (regular excellent finite-dimensional equicharacteristic)  schemes since (to the knowledge of the authors) the Mayer--Vietoris property is not known for the higher Chow groups in this generality.
\end{rema}

\begin{lemma}\label{l4onepo}
Let $X$ and $Y$ be regular schemes, and $x:X\to S$ and $y:Y\to S$ be quasi-projective morphisms, $r,b,c\in \Z$.

Then $x_!(R_X)(b)[2b]\perp y_*(R_Y)(c)[2c+r]$ if $r>0$.

\end{lemma}
\begin{proof}
By   Theorem \ref{pcisdeg}(\ref{itate}) we can assume that $b=0$. 
Next, we have 
$$\dmcs(x_!(R_X),y_*(R_Y)(c)[2c+r])\cong \dmc(X)(R_X,x^!y_*(R_Y)(c)[2c+r])$$ since $x^!$ is left adjoint to $x_!$.

Thus we should prove that 
$$
\dmcx(R_X, x^!y_*(R_Y)(c)[2c+r])=\ns.
$$
We argue somewhat similarly to \cite{brhtp}*{\S2.1}. Let us make certain reduction steps.

Consider a factorization of $x$ as $X\stackrel{f}{\to}S'\stackrel{h}{\to}S$, where $h$ is smooth of dimension $q$, $f$ is an embedding, $S'$ is connected, and consider the corresponding diagram
$$ \begin{CD} 
  Z@>{f_Y}>> Y'@>{h_Y}>> Y\\
  @VV{z_X}V @VV{y'}V@VV{y}V \\
  X@>{f}>> S' @>{h}>> S
\end{CD}$$
 (the upper row is the base change of the lower one to $Y$).
 Then we have $x^!y_*R_Y(\ell)[2\ell]= f^!h^!y_*R_Y(\ell)[2\ell]\cong f^! y'_*h_Y^! R_Y(\ell)[2\ell]$, $\ell=\dim Y$
 (by Theorem \ref{pcisdeg}(\ref{iexch})). Parts \ref{ipur} and \ref{iupstar} of Theorem \ref{pcisdeg} allow us to transform this into 
 $f^! y'_*R_{Y'}$. Hence below we may assume  $x$ to be an embedding (since we can replace $S$ by $S'$ in the assertion). Besides, the isomorphism $x^!y_*\cong z_{X*}z_Y^*$ for $z_Y=h_Y\circ f_Y$ yields that the group in question is zero if $Y$ lies over 
 $S\setminus X$ (considered as a set); see    Theorem \ref{pcisdeg}(\ref{imotcat}).
 
 Next we apply  Lemma \ref{filtr}(1); it yields that   it suffices to verify the statement for $Y$ replaced by the components of some  regular connected stratification. 

Now, we can choose a stratification of this sort such that each  $Y_\ell$  lies either over $X$ or over $S\setminus X$. 
Therefore it suffices to verify our assertion in the case where $y$ factors through $x$. Moreover, since $x^!x_*$ is the identity functor on $\dmc(X)$ (in this case; see  Theorem \ref{pcisdeg}(\ref{iglu})), we may also assume that $X=S$. Next, 
applying the adjunction $y^*\dashv y_*$ we transform $\dms(R_S, y_*(R_Y)(c)[2c+r])$ into $\dmy(y^*(R_S), R_Y(c)[2c+r])= \dmy(R_Y, R_Y(c)[2c+r])$.
Thus it remains to apply the previous Lemma.
\end{proof}

\subsection{Weight structures: short reminder}\label{sws}

We recall some basics of the theory of weight structures.

\begin{definition}\begin{enumerate}
\item A pair of subclasses $\cu_{w\le 0},\cu_{w\ge 0}\subset\obj \cu$ will be said to define a \emph{weight structure} $w$ for $\cu$ if they  satisfy the following conditions:
\begin{enumerate}
\item $\cu_{w\ge 0},\cu_{w\le 0}$ are Karoubi-closed in $\cu$ (i.e., contain all $\cu$-retracts of their objects).

\item {\bf Semi-invariance with respect to translations:} \\
$\cu_{w\le 0}\subset \cu_{w\le 0}[1]$, $\cu_{w\ge 0}[1]\subset\cu_{w\ge 0}$.

\item {\bf Orthogonality:}\\
$\cu_{w\le 0}\perp \cu_{w\ge 0}[1]$.
 
\item {\bf Weight decompositions:}
For any $M\in\obj \cu$ there exists a distinguished triangle 
\begin{equation*}
B\to M\to A\stackrel{f}{\to} B[1]
\end{equation*} 
such that $A\in \cu_{w\ge 0}[1],\  B\in \cu_{w\le 0}$.
\end{enumerate}
\item The category $\hw\subset \cu$ whose objects are $\cu_{w=0}=\cu_{w\ge 0}\cap \cu_{w\le 0}$, $\hw(Z,T)=\cu(Z,T)$ for $Z,T\in \cu_{w=0}$,
will be called the {\it heart} of $w$.

\item $\cu_{w\ge i}$ (resp. $\cu_{w\le i}$, resp. $\cu_{w= i}$) will denote $\cu_{w\ge0}[i]$ (resp. $\cu_{w\le 0}[i]$, resp. $\cu_{w= 0}[i]$).
We denote $\cu_{w\ge i}\cap \cu_{w\le j}$ by $\cu_{[i,j]}$ (so it equals $\ns$ for $i>j$).

\item We will  say that $(\cu,w)$ is {\it  bounded}  if $\cup_{i\in \Z} \cu_{w\le i}=\obj \cu=\cup_{i\in \Z} \cu_{w\ge i}$.

\item Let $\cu$ and $\cu'$ be triangulated categories endowed with weight structures $w$ and $w'$, respectively; let $F:\cu\to \cu'$ be an exact functor.  $F$ will be called {\it left weight-exact} (with respect to $w$, $w'$) if it maps $\cu_{w\le 0}$ into $\cu'_{w'\le 0}$; it will be called {\it right weight-exact} if it maps $\cu_{w\ge 0}$ into $\cu'_{w'\ge 0}$. $F$ is called {\it weight-exact} if it is both left and right weight-exact.

\item Let $H$ be a full subcategory of a triangulated $\cu$.
We will say that $H$ is {\it negative} if $\obj H\perp (\cup_{i>0}\obj (H[i]))$.

\item\label{ifact} We call a category $\frac A B$ a {\it factor} of an additive category $A$ by its (full) additive subcategory $B$ if $\obj \bigl( \frac A B\bigl)=\obj A$ and $\frac A B(M,N)= A(M,N)/(\sum_{O\in \obj B} A(O,N) \circ A(M,O))$.

\end{enumerate}\label{dwstr}\end{definition}

Now we recall those properties of weight structures that will be needed below (and that can be easily formulated). See \cite{brelmot} for the references to the  proofs  (whereas the fact that   $\hw'\cong \frac{\hw}{\hw_{\du}}$in the setting of assertion \ref{iloc} was proved in \cite{bosos}).

\begin{proposition} \label{pbw}
Let $\cu$ be a triangulated category. 
\begin{enumerate}
\item\label{idual}
$(C_1,C_2)$ $(C_1,C_2\subset \obj \cu)$ define a weight structure for $\cu$ if and only if $(C_2^{op}, C_1^{op})$ define a weight structure for $\cu^{op}$.

\item\label{iort} 
Let $w$  be a weight structure for $\cu$. 
Then $\cu_{w\ge 0}=(\cu_{w\le -1})^\perp$ and $\cu_{w\le 0}={}^\perp \cu_{w\ge 1}$ {\nf(}~see \S\ref{snotata}\/{\nf)}.

\item\label{iext} Let $w$  be a weight structure on
 $\cu$. Then  $\cu_{w\le 0}$, $\cu_{w\ge 0}$, and $\cu_{w=0}$
are extension-stable.

\item\label{iuni} 
Suppose that $v$, $w$ are weight structures for $\cu$\/{\nf;} let $\cu_{v\le 0}\subset \cu_{w\le 0}$ and $\cu_{v\ge 0}\subset \cu_{w\ge 0}$.
Then $v=w$ {\nf(}i.e., the inclusions are equalities\/{\nf)}.

\item \label{igen}
Assume  $H\subset \obj \cu$ is negative and   $\cu$ is idempotent complete. 
Then there exists a unique weight structure $w$ on the 
triangulated subcategory $T$ of $\cu$ generated by $H$ such that $H\subset T_{w=0}$. 
Its heart is the {\it envelope} {\nf(}see \S\ref{snotata}{\nf)} of $H$ in $\cu;$ it is the idempotent completion of $H$ if $H$ is additive.
 
\item \label{iwgen} 
For the weight structure mentioned in the previous assertion, $T_{w\le 0}$ is the envelope of  $\cup_{i\le 0}H[i]${\nf;}
$T_{w\ge 0}$ is the envelope of  $\cup_{i\ge 0}H[i]$.

\item \label{iadj} 
Let $\cu$ and $\du$ be triangulated categories endowed with weight structures $w$ and $v$, respectively. 
Let $F: \cu \leftrightarrows \du:G$ be adjoint functors. 
Then $F$ is left weight-exact if and only if $G$ is right weight-exact.


\item\label{iloc} 
Let $w$  be a weight structure for $\cu${\nf;} let $\du\subset \cu$ be a triangulated subcategory of $\cu$. 
Suppose that $w$ yields a weight structure for  $\du$ {\nf(}i.e., $\obj \du\cap \cu_{w\le 0}$ and $\obj \du\cap \cu_{w\ge 0}$ give a weight structure for $\du)$. 

\noindent Then $w$ also induces a weight structure on $\cu/\du$ {\nf(}the localization, i.e., the Verdier quotient of $\cu$ by $\du)$ in the following sense\/{\nf:} the Karoubi-closures of $\cu_{w\le 0}$ and $\cu_{w\ge 0}$ {\nf(}considered as classes of objects of $\cu/\du)$ give a weight structure $w'$ for $\cu/\du$ {\nf(}note that $\obj \cu=\obj \cu/\du)$. 
Besides, 
$\hw'$ is 
naturally equivalent to  $\frac{\hw}{\hw_{\du}}$ if $w$ is bounded.

\item \label{iwegen}
Suppose that $\du\subset \cu$ is a full subcategory of compact objects endowed with a bounded weight structure $w'$. 
Suppose that $\du$ weakly generates $\cu;$ let $\cu$  admit arbitrary {\nf(}small\/{\nf)} coproducts. 
Then $w'$ can be extended to a certain weight structure $w$ for $\cu$. 

\item\label{igluws}
Let $\du\stackrel{i_*}{\to}\cu\stackrel{j^*}{\to}\eu$ be a part of a gluing datum
{\nf(}see Theorem \ref{pcisdeg}\/{\nf(\ref{iglu}))}.
Then for any pair of weight structures on $\du$ and $\eu$ {\nf(}we will denote them by $w_\du$ and $w_\eu$, respectively\/{\nf)}
there exists a weight structure $w$ on $\cu$ such that both $i_*$ and $j^*$ are weight-exact {\nf(}with respect to the corresponding  weight structures\/{\nf)}. 
Besides, the functors $i^!$ and $j_*$ are right weight-exact {\nf(}with respect to the corresponding weight structures\/{\nf);} $i^*$ and $j_!$ are left weight-exact. 
Moreover, 
$$\cu_{w\ge 0}=C_1=\{M\in \obj\cu:\ i^!(M)\in \du_{w_\du\ge 0} ,\ j^*(M)\in \eu_{w_\eu\ge 0} \}$$ and 
$$\cu_{w\le 0}=C_2=\{M\in \obj \cu:\ i^*(M)\in \du_{w_\du\le 0} ,\ j^*(M)\in \eu_{w_\eu\le 0} \}.$$
Lastly, $C_1$ {\nf(}resp. $C_2)$ is the  envelope  of $j_!(\eu_{w\le 0})\cup  i_*(\du_{w\le 0})$  {\nf(}resp. of $ j_*(\eu_{w\ge 0})\cup i_*(\du_{w\ge 0}))$. 

\item\label{igluwsc} 
In the setting of the previous assertion, if $w_{\du}$ and $w_{\eu}$ are  bounded, then $w$ is bounded also. 
 Besides, $\cu_{w\le 0}$ is the envelope of $\{i_*(\du_{w_{\du}=l}),\ j_!(\eu_{w_{\eu}=l}),\ l\le 0\};$
$\cu_{w\ge 0}$ is the envelope of $\{i_*(\du_{w_{\du}=l}),\ j_*(\eu_{w_{\eu}=l}),\ l\ge 0\}$.

\item\label{igluwsn} 
In the setting of  assertion \ref{igluws}, the weight structure $w$ described is the only weight structure for $\cu$ such that both $i_*$ and $j^*$ are weight-exact.
\end{enumerate}
\end{proposition}

\begin{rema}\label{rlift}
Part \ref{iloc} of the proposition can be re-formulated as follows. If $i_*:\du\to \cu$ is an embedding of triangulated categories that is weight-exact (with respect to certain weight structures for $\du$ and $\cu$), an exact functor $j^*:\cu\to \eu$ is equivalent to the localization of $\cu$ by $i_*(\du)$, then there exists a unique weight structure $w'$ for $\eu$ such that the functor $j^*$ is weight-exact. If $w$ is bounded then $\hw_{\eu}$ is 
equivalent to $\frac {\hw} {i_*(\hw_{\du})}$ (with respect to the natural functor $\frac {\hw}{i_*(\hw_{\du})}\to \eu$).
\end{rema}

\section{On the Chow weight structures for relative motives}\label{schow}

This is the main section of our paper. We define the Chow weight structures for  relative motives using the "gluing construction" and study 
their properties. We also prove that the heart of $w_{\chow}(S)$ consists of certain "Chow" motives if $S$ is a variety over a field (or if it is "pro-smooth affine" over a variety).

A substantial part of this section is just a "recombination" of (the corresponding parts of) \cite{brelmot}*{\S2}; yet some of the arguments used in \S\ref{sfwchow}
are quite new (and rather interesting). 

\subsection{The 
construction of the Chow weight structure}\label{schowunr}

First we  describe certain candidates for $\dmcs_{w_{\chow}\ge 0}$ and $\dmcs_{w_{\chow}\le 0}$; 
next we will prove that they yield a weight structure for $\dmcs$ indeed. A reader interested in certain "motivation" for this construction is strongly recommended to look at (the remarks in) \cite{brelmot}*{\S2.3}.

For a scheme $X$ we will denote by $\opx$ (resp. $\onx$) the envelope (see \S\ref{snotata}
) of $p_*(R_P)(s)[i+2s]$($\cong \mgbm_X(P)(s)[i+2s]$; see Remark \ref{ridmot}(2)) 
in $\dmcx$; 
here $p:P\to X$ runs through all  morphisms to $X$ that can be factorized as $g\circ h$, where $h:P\to X'$ is a smooth projective morphism, $X'$ is a regular scheme, $g:X'\to X$ is a finite universal homeomorphism,  $s\in \Z$, whereas $i\ge 0 $ (resp. $i\le 0$). 
We denote $\opx\cap \onx$ by $\ozx$.

\begin{rema}\label{r1fun}
1. Recall that for a morphism $f:Y\to X$ we have 
$f^*\mgbmx(P)\cong \mgbmy({P_Y})$ (see  Remark \ref{ridmot}(2)). 
Next, suppose that for $X'/X$ as above the  scheme $Y'_{red}$ associated with $Y'=X'_Y $ 
 is regular. Then the aforementioned  remark  immediately yields that $f^* 
\mgbmx(P)(s)[2s]\in \oz(Y)$.

Moreover, 
Theorem \ref{pcisdeg} 
easily yields that $f^!\mgbm_X(P)\in \oz(Y)$ if $f$ induces an immersion $Y'_{red}\to X'$ of regular schemes. 
Indeed, consider the diagram \begin{equation}\label{cdred}\begin{CD}
P_{Y,red}@>{p_r}>> P_Y@>{f_P}>> P\\
@VV{h_{Y,red}}V @VV{h_Y}V@VV{h}V \\
Y'_{red}@>{y'_r}>> Y' @>{f'}>> X'\\
@. @VV{g_Y}V@VV{g}V \\
@. Y @>{f}>> X
\end{CD}\end{equation}
where  $p_r$ and $y'_r$ are the corresponding nil-immersions. We can assume that $Y$ and $P$ are connected; hence $Y'_{red}$ and $P_{Y,red}$ are connected also. Denote the codimension of $Y'_{red}$ in $ X'$ by $c$; 
then $p_r\circ f_P:P_{Y,red}\to P$ is an immersion of regular schemes of codimension $c$ also. 
Then 
 we obtain: $f^!\mgbm_X(P)=f^! p_*(R_P)
\cong (g_Y\circ y'_r)_*h_{Y,red_*}(p_r\circ f_P)^!R_{P}$. 
Using  part \ref{ipura} of the theorem, we transform this into $(g_Y\circ y'_r)_*h_{Y,red_*}R_{P_{Y_{red}}}(-c)[-2c] \in \oz(Y)$.

2. Certainly, for $\operatorname{char} X=0$ the universal homeomorphisms mentioned are just isomorphisms.
\end{rema}

For a stratification $\alpha:S=\cup S_\ell^\alpha$ we denote by $\opal$ the class $\{M\in \obj \dmcs:\ j_\ell^! (M)\in \op(S_\ell^\alpha), 1\le \ell\le n\} $;
$\onal=\{M\in \obj \dmcs:\ j_\ell^* (M)\in \on(S_\ell^\alpha), 1\le \ell\le n\} $. 

We define the Chow weight structure for $\dmcs$ as follows: $\dmcs_{w_{\chow}\ge 0}=\cup_\alpha \opal$, $\dmcs_{w_{\chow}\le 0}=\cup_\alpha \onal$; here $\alpha$ runs through all 
stratifications of $S$.

\begin{lemma}\label{lglustr}
1. Let $\delta$ be a {\nf(}not necessarily regular\/{\nf)} stratification of $S;$ we denote the corresponding immersions $S_\ell^\delta\to S$ by $j_\ell$.
Let $M$ be an object of $\dmcs$.
Suppose that $j_\ell^!(M)\in \dmc(S_\ell^\delta)_{w_{\chow}\ge 0}$  {\nf(}resp. $j_\ell^*(M)\in \dmc(S_\ell^\delta)_{w_{\chow}\le 0})$ for all $\ell$.

Then $M\in \dmcs_{w_{\chow}\ge 0}$ {\nf(}resp. $M\in \dmcs_{w_{\chow}\le 0}${\nf)}.

2.   For any immersion $j:V\to S$ we have $j_*(\dmc(V)_{w_{\chow}\ge 0})\subset \dmcs_{w_{\chow}\ge 0}$  and  $j_!(\dmc(V)_{w_{\chow}\le 0})\subset \dmcs_{w_{\chow}\le 0}$.

3. For any $M\in \dmcs_{w_{\chow}\le 0}$ and $N\in \dmcs_{w_{\chow}\ge 1}(=\dmcs_{w_{\chow}\ge 0}[1])$ there exists a stratification $\alpha$ of $S$ such that $M\in \onal$, $N\in \opal[1]$.
\end{lemma}
\begin{proof}
1. We use induction on the number of strata in $\delta$. The $2$-functoriality of motivic upper image functors yields: it suffices to prove the statement for $\delta$ consisting of two strata.

So, let $S=U\cup Z$, $U$ and $Z$ are disjoint, $U\neq \ns$ is open  in $S$; 
we denote the immersions $U\to S$ and $Z\to S$ by $j$ and $i$, respectively.
By the assumptions on $M$, there  exist stratifications $\beta$ of $Z$ and $\gamma$ of $U$ such that $i^!(M)\in \op(\beta)$ and $j^!(M)\in \op(\gamma)$ (resp. $i^*(M)\in \on(\beta)$ and $j^*(M)\in \on(\gamma)$). 

We take the union of $\beta$ with $\gamma$ and denote by $\alpha$ the stratification of $S$  obtained  (for $\# \gamma=\Gamma$ we put $S_\ell^\alpha=U_\ell^\gamma$ if $1\le \ell\le \Gamma$ and $S_\ell^\alpha=Z_{\ell-\Gamma}^\beta$ if $\ell> \Gamma$; note that we really obtain a stratification in our weak sense of this notion this way; see \S\ref{snotata}).
Then  the $2$-functoriality of $-^!$ (resp. of $-^*$) yields that $M\in \opal$ (resp. $M\in \onal$).

2. We choose a   stratification $\delta$ containing $V$ (as one of the strata). So we assume that $V=S_v^\delta$ for some index $v$. Then it can be easily seen that $j_{u}^!j_{v*}^{\vphantom !}=0=j_{u}^*j_{v!}^{\vphantom *}$  for any $u\neq v$ and  $j_v^!j_{v*}^{\vphantom !}\cong 1_{\dm(V)} \cong j_v^*j_{v!}^{\vphantom *}$ (see    Theorem \ref{pcisdeg}(\ref{iglu})). Hence the result follows from assertion 1.

3. By Remark \ref{r1fun}(1) it suffices to verify the following: if $\beta$, $\gamma$ are stratifications of $S$, and  $S_{i\ell}\to S_\ell^{\beta}$, $S'_{i\ell}\to S_\ell^{\gamma}$ are  (finite) sets of  finite universal homomorphisms, then there exists  a common subdivision $\alpha$ of $\beta$ and $\gamma$ such that all the (reduced) schemes $(S_{i\ell}\times_S {S_{m}^{\alpha}})_{red}, (S'_{i\ell}\times_S {S_{m}^{\alpha}})_{red}$ are regular. To this end it obviously suffices to prove the following: if $f:Z\to S$ is an immersion, $g_i:T_i\to Z$ are some finite universal homeomorphisms, then there exists a stratification $\delta$ of $Z$ such that the schemes $T_{i\ell}=(T_i\times_Z Z_\ell^{\delta})_{red}$ are regular for all $i$ and $\ell$.

We prove the latter statement by 
 Noetherian induction. Suppose that it is fulfilled for any proper closed subscheme $Z'$ of $Z$. Since all $(T_i)_{red}$ are generically regular, we can choose a (sufficiently small) open non-empty subscheme $Z_1$ of $Z$ such that all of  $(T_{i}\times_Z Z_1)_{red}$ are regular. 

Next, apply the inductive assumption to the scheme $Z'=Z\setminus Z_1$ and the morphisms $g_i'=g_i\times _Z Z'$; we choose a  stratification $\alpha'$ of $Z'$ such that all $T'_{i\ell}=(T_i\times_Z {Z'}_\ell^{\alpha'})_{red}$ are regular. Then it remains to take the union of $Z_1$ with $\alpha'$, i.e., we consider the following stratification $\alpha$: $Z_1^{\alpha}=Z_1$, and $Z_\ell^{\alpha}={Z'}_{\ell-1}^{\alpha'}$ for all $\ell>1$.   
 \end{proof}

\begin{theorem}\begin{enumerate}
\item The couple $(\dmcs_{w_{\chow}\ge 0},\, \dmcs_{w_{\chow}\le 0})$ yields a bounded weight structure $w_{\chow}$ for $\dmcs$.  

\item $\dmcs_{w_{\chow}\ge 0}$ {\nf(}resp. $\dmcs_{w_{\chow}\le 0})$ is the envelope of $p_*(R_P)(s)[2s+i]$ {\nf(}resp. of $\mgbm_S(P)(s)[2s-i])$ for $s\in\Z$, $i\ge 0$, and $p:P\to S$ being the composition of a smooth projective morphism with a finite universal homeomorphism whose domain is regular and with an immersion. 
 
\item $w_{\chow}$ can be extended to a weight structure $\wchowb$ for the whole $\dms$.
\end{enumerate}\label{pwchowa}\end{theorem}

\begin{proof} 
1-2. We prove the statement by Noetherian induction. So, we suppose that assertions 1 and 2 are fulfilled for all proper closed subschemes of $S$. We prove them for $S$.

We denote the envelopes mentioned in assertion 2 by $(\dmcs_{w_{\chow}'\ge 0},\dmcs_{w_{\chow}'\le 0})$. We should prove that $w_{\chow}$ and $w_{\chow}'$ yield coinciding weight structures for $\dmcs$.

Obviously,  $\dmcs_{w_{\chow}\le 0}$, $\dmcs_{w_{\chow}\ge 0}$, $\dmcs_{w_{\chow}'\le 0}$, and $\dmcs_{w_{\chow}'\ge 0}$ are Karoubi-closed in $\dmcs$, and are semi-invariant with respect to translations (in the appropriate sense).

Now, Lemma \ref{lglustr}(2) yields that $\dmcs_{w_{\chow}'\le 0}\subset \dmcs_{w_{\chow}\le 0}$ and $\dmcs_{w_{\chow}'\ge 0}\subset \dmcs_{w_{\chow}\ge 0}$. Hence in order to check that $w_{\chow}$ and $w_{\chow}'$ are indeed weight structures, it suffices to verify the following: 

(i) the orthogonality axiom for $w_{\chow}$;

(ii) any  $M\in \obj\dmcs$ possesses a weight decomposition with respect to $w_{\chow}'$.

Thus these statements along with the boundedness of $w_{\chow}$ imply assertion 1.
Besides, Proposition \ref{pbw}(\ref{iuni}) yields that these two statements imply assertion 2 also, whereas in order to prove  assertion 1 it suffices to verify the boundedness of $w_{\chow}'$ (instead of that for $w_{\chow}$). 

Now we verify (i). For some fixed $M\in \dmcs_{w_{\chow}\le 0}$ and $N\in \dmcs_{w_{\chow}\ge 1}$ we check that $M\perp N$. 
By Lemma \ref{lglustr}(3), we can assume that $M\in \onal$, $N\in \opal[1]$ for some stratification $\alpha$ of $S$.
Hence it suffices to prove  that $\onal\perp \opal[1]$ for any $\alpha$.
The latter statement is an easy consequence of  Lemmas \ref{l4onepo} and \ref{filtr}(1).

Now we verify (ii) along with the boundedness of $w_{\chow}'$.
We choose some generic point $K$ of $S$, denote by $K^p$ its perfect closure, and by $j_{K^p}:K^p\to S$ the corresponding morphism.  
We fix some $M$. 
By Theorem \ref{pcisdeg}(\ref{igenc}), 
there exist some smooth projective varieties $P_{i}/K^p$, $1\le i\le n$,
(we denote the corresponding morphisms $P_{i}\to K^p$ by $p_i$) and some $s\in \Z$ such that $j_{K^p}^*(M)$ belongs to the triangulated subcategory of $\dmc(K^p)$ generated by  $\{p_{i*}(R_{P_{i}})(s)[2s]\}$.
Now we choose some  finite  universal homeomorphism $K'\to K$ (i.e., a morphism of spectra of fields corresponding to a finite  purely inseparable extension) such that $P_{i}$ are defined (and are smooth projective) over $K'$. 
By Theorem \ref{pcisdeg}(\ref{icont},\ref{itr}), for the corresponding morphisms $j_{K'}:K'\to S$ and $p_i':P_{K',i}\to K'$ we have the following: $j_{K'}^*(M)$ belongs to the
triangulated subcategory of $\dmc(K')$ generated by  $\{p'_{i*}(R_{P_{K',i}})(s)[2s]\}$. 
Applying Zariski's main theorem in the form of Grothendieck, we can choose a finite universal homeomorphism $g$ from a regular scheme $U'$ whose generic fibre is $K'$ to an open $U\subset S$ ($j:U\to S$ will denote the corresponding immersion) and  smooth projective $h_i:P_{U',i}\to U'$ such that the fibres of $P_{U',i}$ over $K'$ are isomorphic to $P_{K',i}$. Moreover, by Theorem \ref{pcisdeg}(\ref{icont}) we can also assume that $(j\circ g)^*(M)$ belongs to the
triangulated subcategory of $\dmc(U')$ generated by  $\{h_{i*}(R_{P_{U',i}})(s)[2s]\}$.
Then Theorem \ref{pcisdeg}(\ref{itr}) yields that $j^*(M)$ belongs to the triangulated subcategory $D$ of $\dmc(U)$ generated by  $\{(g\circ h_i)_{*}(R_{P_{U',i}})(s)[2s]\}$.

Since $\id_U$ yields a  stratification of $U$,  the set $\{(g\circ h_i)_{*}(R_{P_{U',i}})(s)[2s]\}$  is negative in $\dmc(U)$ (since $\onal\perp \opal[1]$ for any  $\alpha$, as we have just proved). 
Therefore (by Proposition \ref{pbw}(\ref{igen}--\ref{iwgen})) there exists a weight structure $d$ for $D$ such that $D_{d\ge 0}$ (resp. $D_{d\le 0}$) is the envelope of $\cup_{n\ge 0} \{(g\circ h_i)_{*}(R_{P_{U',i}})(s)[2s+n] \}$ (resp. of $\cup_{n\ge 0} \{(g\circ h_i)_{*}(R_{P_{U',i}})(s)[2s-n]\}$). 
We also obtain that $D_{d\ge 0}\subset \dmc(U)_{w_{\chow}'\ge 0}$ and $D_{d\le 0}\subset \dmc(U)_{w_{\chow}'\le 0}$. 

We denote $S\setminus U$ by $Z$ ($Z$ may be empty); $i:Z\to S$ is the corresponding closed immersion.
 By the inductive assumption, $w_{\chow}$ and $w_{\chow}'$ yield coinciding  bounded weight structures for $\dmc(Z)$. 

We have a gluing datum $\dmc(Z)\stackrel{i_*}{\to}\dmcs \stackrel{j^*}{\to} \dmc(U)$.
We can 'restrict it' to a gluing datum 
$$
 \dmc(Z)\stackrel{i_*}{\to} j^*{\ob} (D)\stackrel{j_0^*}{\to} D 
$$ 
(see  Proposition \ref{pbw}(\ref{igluws})), whereas $M\in \obj (j^*{\ob}(D))$; here $j_0^*$ is the corresponding restriction of $j^*$. 
Hence by Proposition \ref{pbw}(\ref{igluws}) there exists a weight structure $w'$ for $j^*{\ob}(D)$ such that the functors $i_*$ and $j_0^*$ are weight-exact (with respect to the weight structures mentioned). 
Hence there exists a weight decomposition $B\to M\to A$ of $M$ with respect to $w'$. 
Besides, there exist $m,n\in \Z$ such that $j^*_0(M)\in \dmc(U)_{w_{\chow}'\ge m}$, $j^*_0(M)\in \dmc(U)_{w_{\chow}'\le n}$, $i^!(M)\in \dmc(Z)_{w_{\chow}'\ge m}$, and $i^*(M)\in \dmc(Z)_{w_{\chow}'\le n}$.
 Hence $A[-1],M[-m]\in \dmcs_{w_{\chow}'\ge 0}$; $B, M[-n]\in \dmcs_{w_{\chow}'\le 0}$; 
 here we apply Proposition \ref{pbw}(\ref{igluwsc}). 
So, we have verified (ii) and the boundedness of $w_{\chow}'$. 
As was shown above, this finishes the proof of assertions 1-2.

3. Since ${\underline{Hw}}_{\chow}$ generates $\dmcs$, and $\dmcs$ weakly generates $\dms$ (by 
Theorem \ref{pcisdeg}(\ref{imotgen})), 
${\underline{Hw}}_{\chow}$ weakly generates $\dms$. 
Hence the assertion follows immediately from assertion 1 and Proposition \ref{pbw}(\ref{iwegen}). 
\end{proof}

\subsection{
The main properties of $w_{\chow}(-)$}\label{sfwchow} 

Now we study (left and right) weight-exactness of the motivic image functors. 

\begin{theorem}
\begin{enumerate}
\item The functors $-(b)[2b](=-\otimes R_S(b)[2b])$ 
are weight-exact with respect to $w_{\chow}$ for all $S$ and all $b\in \z$.
\item Let $f:X\to Y$ be a 
 {\nf(}separated\/{\nf)} morphism of finite type.
\begin{enumerate}
\item The functors $f^!$ and $f_*$ are right weight-exact{\nf\/;} $f^*$ and $f_!$ are left weight-exact.
\item Suppose moreover that $f$ is smooth. Then $f^*$ and $f^!$ are also weight-exact.
\item If $S_{red}$ is regular then $R_S\in \dmcs_{w_{\chow}=0}$.

\item $\dmcs_{w_{\chow}\le 0}$ is the envelope of $\mgbms(T)(b)[2b-r]$ for $T$ running through all schemes of finite type over $S$, $b\in \z$, $r\ge 0;$ 
$\dmcs_{w_{\chow}\ge 0}$ is the envelope of $t_*(R_T)(b)[2b+r]$ for $t:T\to S$ running through all morphisms of finite type with {\bf regular} domains, 
$b\in \z$, $r\ge 0$.

\item Besides, the functor $g^*$ is right weight-exact for $g$ being an arbitrary {\nf(}separated, not necessarily of finite type\/{\nf)} morphism of schemes. It is weight-exact if 
  $g$ is either {\nf(i)}  a {\nf(}filtering\/{\nf)} projective limit of smooth morphisms such that the corresponding connecting morphisms are smooth affine  or {\nf(ii)} 
	a finite universal homeomorphism. In the latter case 
	$g^!$ is weight-exact also.
\end{enumerate}
\item  Let $i:Z\to X$ be a closed immersion{\nf;} let $j:U\to X$ be the complementary open immersion.
\begin{enumerate}
\item $\hwchow(U)$ is 
equivalent to the factor {\nf(}in the sense of  Definition {\nf\ref{dwstr}(\ref{ifact}))} of $\hwchow(X)$ by $i_*(\hwchow(Z))$.  
\item For $M\in \obj \dmcx$ we have{\nf:} $M\in \dmcx_{w_{\chow}\ge 0}$ {\nf(}resp. $M\in \dmcx_{w_{\chow}\le 0}${\nf)} if and only if $j^!(M)\in \dmc(U)_{w_{\chow}\ge 0}$ and $i^!(M)\in \dmc(Z)_{w_{\chow}\ge 0}$ {\nf(}resp.  $j^*(M)\in \dmc(U)_{w_{\chow}\le 0}$ and $i^*(M)\in \dmc(Z)_{w_{\chow}\le 0})$.
\end{enumerate}
\item Let $S=\cup S_\ell^\alpha$ be a stratification, $j_\ell:S_\ell^\alpha\to S$ be the corresponding immersions. Then for $M\in \obj \dmcs$ we have: $M\in \dmcs_{w_{\chow}\ge 0}$ {\nf(}resp. $M\in \dmcs_{w_{\chow}\le 0})$ if and only if $j_\ell^!(M)\in \dmc(S_\ell^\alpha)_{w_{\chow}\ge 0}$ {\nf(}resp. $j_\ell^*(M)\in \dmc(S_\ell^\alpha)_{w_{\chow}\le 0})$ for all $\ell$.
\item 
 For any $S$ we have $R_S\in \dmcs_{w_{\chow}\le 0}$.
\end{enumerate}\label{tfunctwchow}\end{theorem}
\begin{proof}

1. 
Immediate from Theorem \ref{pwchowa}(2).

2. 
Let $f$ be an immersion. Then the description of $w_{\chow}(-)$ given by Theorem \ref{pwchowa}(2) yields that $f_*$ is left weight-exact and $f_!$ is right  weight-exact.
Hence the corresponding adjunctions yield 
(by Proposition \ref{pbw}(\ref{iadj})) that $f^!$  is left weight-exact and $f^*$ is right weight-exact.

If $f$ is smooth, then Theorem \ref{pwchowa}(2) (along with Theorem \ref{pcisdeg}(\ref{iexch}))  easily yields that $f^*$ is left-weight exact and $f^!$ is right weight-exact (since schemes that are smooth over regular bases are regular themselves). 
Hence part \ref{ipur} of  the theorem (along with assertion 1) implies that both of these functors are weight-exact  (so, we obtain assertion 2b).
Besides,  adjunctions yield (by part (\ref{iadj}) of Proposition \ref{pbw})  that $f_!$  is left weight-exact and $f_*$ is right weight-exact.

Thus assertion 2a is valid for any quasi-projective $f$ (since such an $f$ can be presented as  the composition of a closed immersion with a smooth morphism).

Now we verify assertion 2c. Let $S_{red}$ be a regular scheme; denote by $v$ the canonical immersion $S_{red}\to S$.
Then $v_*(R_{{S_{red}}})\in \dmcs_{w_{\chow}=0}$ by 
Theorem \ref{pwchowa}(2). 
Since $v_*(R_{{S_{red}}})\cong R_S$ by Theorem \ref{pcisdeg}(\ref{itr}), we obtain the result.  


Now we are able to prove assertion 2d. First we note (using  Theorem \ref{pwchowa}(2))  that $\dmcs_{w_{\chow}\le 0}$  and $\dmcs_{w_{\chow}\ge 0}$ are subclasses of the corresponding envelopes. 
So, we should verify the converse inclusions.  Note that 
 any excellent Noetherian scheme admits a stratification the reductions of whose components are regular. Hence Lemma \ref{filtr}(2) yields that it suffices to check the following: if  $T$ is a regular  scheme of finite type over $S$, $b\in \z$, and $r\ge 0$, then 
  $\mgbms(T)(b)[2b-r]\in \dmcs_{w_{\chow}\le 0}$ and 
$t_*(R_T)(b)[2b+r] \in \dmcs_{w_{\chow}\ge 0}$. Applying the lemma once more, we reduce the first of these inclusion statements to the case where $T$ is quasi-projective over $S$ (since any scheme of finite type over $S$  possesses a stratification whose components are quasi-projective over $S$).  Similarly, part 3 of the lemma 
allows us to assume that $T$ is (regular and) quasi-projective over $S$ in the second of these inclusion statements. Hence it suffices to note that $R_T\in \dmc(T)_{w_{\chow}=0}$ (by assertion 2c of our Theorem), and apply our assertion 1 along with assertion 2a (for the quasi-projective morphism $t$).

Now we return to the proof of assertion 2a for a general  $f$ (of finite type). Assertion 2d immediately yields the  left weight-exactness of $f_!$. Along with  
Theorem \ref{pcisdeg}(\ref{iexch}) it also easily yields  the  left weight-exactness of $f^*$. Lastly, $f^!$ and $f_*$ are right weight-exact by  Proposition \ref{pbw}(\ref{iadj}).

The first statement in assertion 2e (also) easily follows from assertion 2d (along with Remark \ref{ridmot}(2)). 
Assertion 2d (along with Remark \ref{ridmot}(1)) also implies the weight-exactness of $g^*$ in case (i) (since pro-smooth limits of regular schemes are regular). Lastly, $g^*\cong g^*$ if $g$ is a finite universal homeomorphism (see Theorem \ref{pcisdeg}(\ref{itr})); this finishes the proof of the assertion.

3. Since $i_*\cong i_!$ in this case, the functor $i_*$ is weight-exact by assertion 2a. The functor $j^*$ is weight-exact by assertion 2b.

a). $\dmc(U)$ is the localization of $\dmc(X)$ by $i_*(\dmc(Z))$ by  Theorem \ref{pcisdeg}(\ref{igluc}). 
Hence Proposition \ref{pbw}(\ref{iloc}) yields the result (see Remark \ref{rlift}; cf. also \cite{wildic}*{Theorem 1.7}).

b).  Theorem \ref{pcisdeg}(\ref{igluc}) yields that $w_{\chow}(X)$ is exactly the weight structure obtained by 'gluing $w_{\chow}(Z)$ with $w_{\chow}(U)$' via   Proposition \ref{pbw}(\ref{igluws}) (here we use Theorem \ref{pcisdeg}(\ref{igluwsn})). So we obtain the assertion desired (note that $j^*=j^!$). 

4. The assertion can be easily proved by induction on the number of strata  using assertion 3b.

5. 
Immediate from assertion 2c.
\end{proof}

\begin{rema}\label{rqp}
1. Theorem \ref{pwchowa}(2) and assertion 2d of the previous theorem   give two distinct descriptions of $(\dmcs_{w_{\chow}\ge 0}, \dmcs_{w_{\chow}\le 0})$ as certain envelopes.
It
 certainly follows that instead of all $T$ considered in the assertion mentioned it suffices to take only those $T$ that are quasi-projective over $S$.  

2. One may apply the argument used in the proof of \cite{jin}*{Lemma 2.23} to show that $\mgbms(X)\otimes \mgbms(Y)\cong \mgbms(X\times Y)$  for $X$ and $Y$ being any  schemes of finite type over $S$ (note that loc. cit. itself gives this statement for $R=\Q$). It certainly follows that $\dmcs_{w_{\chow}\le 0}\otimes \dmcs_{w_{\chow}\le 0}\subset \dmcs_{w_{\chow}\le 0}$. 
\end{rema}

Now we prove that positivity and negativity of objects of $\dmcs$ (with respect to $w_{\chow}$) can be 'checked at points'; this is a motivic analogue of  \cite{bbd}*{\S5.1.8}.

\begin{proposition}\label{ppoints}
Let ${\mathcal{S}}$  denote the set of {\nf(}Zariski\/{\nf)} points of $S;$ for  $K\in {\mathcal{S}}$ we will denote the corresponding morphism $K\to S$ by $j_K$.

Then $M\in \dmcs_{w_{\chow}\le 0}$ {\nf(}resp. $M\in \dmcs_{w_{\chow}\ge 0})$ if and only if for any $K\in {\mathcal{S}}$ we have $j_K^*(M)\in \dmc(K)_{w_{\chow}\le 0}$ {\nf(}resp. $j_K^!(M)\in \dmc(K)_{w_{\chow}\ge 0});$ see Remark \ref{ridmot}(1).
\end{proposition}

\begin{proof}
If $M\in \dmcs_{w_{\chow}\le 0}$ (resp. $M\in \dmcs_{w_{\chow}\ge 0}$) then Theorem \ref{tfunctwchow}(2e) (along with (2a))
yields that $j_K^*(M)\in \dmc(K)_{w_{\chow}\le 0}$ (resp. $j_K^!(M)\in \dmc(K)_{w_{\chow}\ge 0}$) indeed.

We prove the converse implication by Noetherian induction. So, suppose that our assumption is true for motives over any closed subscheme of $S$, and that for some $M\in \obj\dmcs$ we have  $j_K^*(M)\in \dmc(K)_{w_{\chow}\le 0}$ (resp. $j_K^!(M)\in \dmc(K)_{w_{\chow}\ge 0}$) for any  $K\in {\mathcal{S}}$.

We should prove that $M\in \dmcs_{w_{\chow}\le 0}$ (resp. $M\in \dmcs_{w_{\chow}\ge 0}$).
By Proposition \ref{pbw}(\ref{iort}) it suffices to verify:  for any  $N\in \dmcs_{w_{\chow}\ge 1} $ (resp. for any $N\in \dmcs_{w_{\chow}\le -1}$), and any $h\in \dmcs(M,N)$  (resp. any  $h\in \dmcs(N, M)$) we have $h=0$. 
We fix some $N$ and $h$.

By the 'only if' part of our assertion (that we have already proved) we have $j_K^*(N)\in  \dmc(K)_{w_{\chow}\ge 1}$  (resp. $j_K^*(N)\in  \dmc(K)_{w_{\chow}\le -1}$); hence $j_K^*(h)=0$. By Theorem \ref{pcisdeg}(\ref{icont}) we obtain that $j^*(h)=0$ for some open embedding $j:U\to S$, where $K$ is a generic point of $U$.

Now suppose that $h\neq 0$; let $i:Z\to S$ denote the closed embedding that is complementary to $j$. 
Then Lemma \ref{filtr}(2) yields that $\dmcs (i^*(M),i^!(N))\neq \ns$  (resp. $\dmcs (i^*(N),i^!(M))\neq \ns$). 
Yet $i^!(N)\in \dmc(Z)_{w_{\chow}\ge 1}$  (resp. $i^*(N)\in \dmc(Z)_{w_{\chow}\le -1}$) by Theorem \ref{tfunctwchow}(b2), whereas $i^* (M)\in \dmc(Z)_{w_{\chow}\le 0}$  (resp. $i^! (M)\in \dmc(Z)_{w_{\chow}\ge 0}$) by the inductive assumption. 
The contradiction obtained proves our assertion.
\end{proof}

Lastly we prove that 'weights are continuous'.

\begin{lemma}\label{lwcont}
Let $K$ be a generic point of  $S;$ denote the morphism $K\to S$ by $j_K$.
Let $M$ be an object of $\dmcs$, and suppose that $j_K^*M\in \dmck_{w_{\chow}\ge 0}$ \nf{(}resp.  $j_K^*M\in\dmck_{w_{\chow}\le 0})$. 
Then there exists an open immersion $j:U\to S$, $K\in U$, such that $j^*M\in \dmc(U)_{w_{\chow}\ge 0}$ \nf{(}resp. $j^*M\in \dmc(U)_{w_{\chow}\le 0})$.
\end{lemma}

\begin{proof}
First we treat the case $j_K^*(M)\in \dmck_{w_{\chow}\ge 0}$.
We consider a weight decomposition of $M[1]$: $B\stackrel{g}{\to} M[1]{\to} A\to B[1]$.
We obtain that $j^*_K(g)=0$ (since $\dmck_{w_{\chow}\le 0}\perp j^*_K(M)[1])$; hence (by Theorem \ref{pcisdeg}(\ref{icont}))  there exists an open immersion $j:U\to S$ ($K\in U$)
such that $j^*(g)=0$. Thus $j^*M[1]$ is a retract of $j^*A$. Since $j^*A[-1]\in \dmc(U)_{w_{\chow}\ge 0}$ (see Theorem \ref{tfunctwchow}(2b)), and $\dmc(U)_{w_{\chow}\ge 0}$ is Karoubi-closed in $\dmc(U)$, we obtain the result.
 
The second part of our statement (i.e., the one for the case  $j_K^*(M)\in\dmck_{w_{\chow}\le 0}$) can be easily verified using the dual argument (see Proposition \ref{pbw}(\ref{idual})). 
\end{proof}

\subsection{Describing $\hwchow$ via Chow motives}\label{stmain}

Now we prove that  for certain $S$ the heart of $w_{\chow}(S)$ has a 
quite  "explicit" description in terms of certain Chow motives over $S$ (whence the name).

\begin{rema}
For a scheme $S$ define the category $\chows$ of Chow motives over $S$ as the Karoubi-closure of $\{\mgbm_S(X)(r)[2r]\}=\{f_*(R_X)(r)[2r]\}$ in $\dmcs$; here $f:X\to S$ runs through all  
 proper morphisms such that $X$ is regular, $r\in \Z$.

Then Theorem \ref{tfunctwchow}(2c,2a) yields that  $\chows\subset \hwchow(S)$.

\end{rema}

Now we prove that in some cases the latter embedding is an equivalence of categories. 

\begin{proposition}\label{pchow}
 Assume that $S$ can be presented as a filtered {\nf(}projective\/{\nf)} limit of varieties over some {\nf(}not necessarily prime\/{\nf)} field  with smooth and affine transition morphisms.
Then $\chows= \hwchow(S)$.
\end{proposition}
\begin{proof}
Since $\chows\subset \hwchow(S)$, it is negative (by the orthogonality axiom of weight structures). Hence Proposition \ref{pbw}(\ref{igen}) yields the existence
of a weight structure $w$ on the triangulated subcategory $T$  of $\dmcs$ generated by $\chows$ 
 such that $\chows\cong \hw$.
Besides, parts \ref{iext} and \ref{iwgen} of the proposition yield that the embedding  of  $T$ into $ \dmcs$ is weight-exact (with respect to $w$ and $w_{\chow}$). 
Hence part \ref{iuni} of the proposition reduces the assertion to the fact that $T=\dmcs$, i.e., that  $\chows$ generates $\dmcs$.

By Theorem \ref{pcisdeg}(\ref{igenc}) the latter is true if $S$ is a variety itself. Hence it remains to pass to the ("pro-smooth affine") limit in this statement.
By continuity (Theorem \ref{pcisdeg}(\ref{icont}) in order to achieve  it suffices to note that pro-smooth affine morphisms respect Chow motives; the latter is immediate from Remark \ref{ridmot}(2) (along with the fact that pro-smooth base change preserves regularity of schemes).
\end{proof}

\begin{rema}\label{rchow}
1. This argument also shows that we could have considered only projective (regular) $X/S$ in the definition of $\chow(S)$; we would have obtained the same category $\chow(S)$ (at least) when $S$ is as in Proposition \ref{pchow}.

2. Actually, the negativity of $\chow(S)$ in $\dmcs$ (for the "projective version" of the definition)
follows immediately from Lemma \ref{l4onepo}. Thus we could have defined the corresponding "restriction" of $w_{\chow}$ (on a full subcategory of $\dmcs$) without any gluing arguments. Next, restricting ourselves to the case where $\dmcs$ is "Chow-generated", we could have easily applied the arguments used in the proof of  Theorem \ref{tfunctwchow} to this version of the Chow weight structure. This is  (basically) the approach to the study of the Chow weight structures used in
 \cite{hebpo} and \cite{brelmot}*{\S2.1--2.2} (yet some of the methods used in the proof of  Theorem \ref{tfunctwchow} are "newer"; they were developed in \cite{brhtp} and in the current paper). 

We chose not to apply this approach in the current paper since the class of base schemes for which we can use it is too small. The reason for this is that for 
$\Q$-linear motives one only needs certain alterations for the 
corresponding analogue of Theorem \ref{pcisdeg}(\ref{igenc}) (cf. \cite{degcis}*{\S4.1}), whereas for
$\zop$-linear motives one requires the so-called prime-to-$l$ alterations (for all primes $l\neq p$; cf. the proof of \cite{cdintmot}*{Proposition 7.2}) whose existence is only known in the context the aforementioned assertion.

3. Certainly, our proposition does not yield a "full description" of $\hwchow(S)$ since we have not "computed the morphisms" in $\chow(S)$. We note that  the argument used in (the proof of)  \cite{brelmot}*{Lemma 1.1.4(I.1)} allows one  to  
express $\dmcs(\mgbm_S(X)(r)[s], \mgbm_S(X')(r')[s'])$ in terms of certain (Borel-Moore) motivic homology groups; here $r,s,r',s'\in \z$, $X$ and $X'$ are regular schemes that are projective over $S$ (whereas $S$ is "pro-smooth affine" over a variety). Thus we can compute a "substantial part" of morphism groups in  $\chow(S)$. 
Yet computing the composition operation for $\chow(S)$-morphisms is a much more difficult problem; for $R=\Q$ it was recently solved in \cite{jin}.
\end{rema}

\section{Applications to (co)homology of motives and other matters}\label{sapcoh}
In this section we list some immediate applications for our results (following \cite{brelmot}; so a reader well acquainted with ibid. may skip this section completely or just have a look at Proposition \ref{pscholl}). Most of the statements below easily follow from the results of \cite{bws}; there is absolutely no problem to apply the corresponding arguments from  \cite{brelmot}*{\S3} (where the case $R=\mathbb{Q}$ was considered) for their proofs. 
The results seem to be "more interesting" in the case where $S$ is a pro-smooth affine limit of varieties (cf. Proposition \ref{pchow}).
 
\subsection{The weight complex functor for $\dmcs$ and its Grothendieck group} \label{swc}

We note that the weight complex functor (whose 'first ancestor' was defined 
in \cite{gs}) can be defined for $\dmcs$.

\begin{proposition}\label{pwc}\begin{enumerate}
\item The embedding $\hwchows\to K^b(\hwchows)$ factors through a certain  {\it weight complex} functor $t_S:\dmcs\to K^b(\hwchows)$ which is exact and conservative.
\item For $M\in \obj \dmcs$, $i,j\in \Z$, we have $M\in \dmcs_{[i,j]}$ {\nf(}see Definition \ref{dwstr}{\nf(4))} if and only if  $t_S(M)\in K(\hwchows)_{[i,j]}$. 
\end{enumerate}\end{proposition}


Now we calculate $K_0(\dmcs)$ and define a certain Euler characteristic for schemes that are  of finite type (and separated) over $S$.

\begin{proposition}\label{pkz}\begin{enumerate}
\item We define $K_0(\hwchows)$  as the abelian group whose generators are $[M]$, $M\in  \dmcs_{w_{\chow}=0}$, and the relations are $[B]=[A]+[C]$ if
$A,B,C\in  \dmcs_{w_{\chow}=0}$ and $B\cong A\bigoplus C$.
For $K_0(\dmcs)$ we take similar generators and  set $[B]=[A]+[C]$ if $A\to B\to C\to A[1]$ is a distinguished triangle.

\noindent
Then the embedding $\hwchows\to \dmcs$  yields an isomorphism $K_0(\hwchows)\cong K_0(\dmcs)$.

\item For the correspondence $\chi:X\mapsto [\mgbms(X)]$ 
 from the class of schemes  separated of finite type over $S$ to $K_0(\dmcs)\cong K_0(\hwchows)$ we have: $\chi(X\setminus Z)=\chi(X)-\chi(Z)$ if $Z$ is a closed subscheme of $X$.
\end{enumerate}
\end{proposition}

\subsection{On Chow-weight spectral sequences and filtrations}\label{schws}

Now we discuss (Chow)-weight spectral sequences and filtrations for 
cohomology of motives.  Certainly, here one can pass to homology via obvious dualization (see Proposition\ref{pbw}(\ref{idual})).
We note that any weight structure yields certain weight spectral sequences for any (co)homology theory; the main distinction of the result below from the general case (i.e., from \cite{bws}*{Theorem 2.4.2}) is that $T(H,M)$ always converges  (since  $w_{\chow}$ is bounded). 

\begin{proposition}\label{pwss}
Let $\underline{A}$ be an abelian category.\begin{enumerate}
\item Let $H:\dmcs\to \underline{A}$ be a cohomological functor{\nf\/;} for any $r\in \Z$ denote $H\circ [-r]$ by $H^r$.
For an $M\in \obj\dmcs$ we denote by $(M^i)$ the terms of $t_S(M)$ {\nf(}so, $M^i\in  \dmcs_{w_{\chow}=0};$ here we can take any possible choice of $t_S(M))$. 

Then the following statements are valid.
\begin{enumerate}
\item There exists a \emph{(Chow-weight)}  spectral sequence $T=T(H,M)$ with $E_1^{pq}=H^q(M^{-p})\implies H^{p+q}(M);$ the differentials for $E_1(T(H,M))$ come from $t_S(M)$.
\item $T(H,M)$ is $\dmcs$-functorial in $M$ {\nf(}and does not depend on any choices{\nf)} starting from $E_2$.
\end{enumerate}
\item Let $G:\dmcs\to \underline{A}$ be any contravariant functor.
Then for any $m\in \Z$ the object $(W^{m}G)(M)=\imm (G(w_{\chow,\ge m}M)\to G(M))$ does not depend on the choice of $w_{\chow,\ge m}M;$ it is functorial in $M$.
\noindent
We call the filtration of $G(M)$ by $(W^{m}G)(M)$ its {\it Chow-weight} filtration. If $G$ is cohomological, it coincides with the filtration given by $T(G,M)$.

\end{enumerate}\end{proposition}

\begin{rema}\label{rintel}
1.  We obtain certain ("motivically functorial") {\it Chow-weight} spectral sequences and
filtrations for any (co)homology of motives. In particular, we have them
for  \'etale and motivic cohomology of $S$-motives (with coefficients in an $R$-algebra).
 The corresponding functoriality results  cannot be proved using 'classical' (i.e., Deligne's) methods, since the latter heavily rely on the degeneration of (an analogue of) $T$ at $E_2$. 

On the other hand, we probably do not have the "strict functoriality" property for this filtration
unless $T(H,M)$ degenerates for any $M\in \obj \dmcs$ (see \cite{btrans}*{Proposition 3.1.2(II)}), 
whereas this degeneration is only known to hold for "more or less classical" $\mathbb{Q}$-linear cohomology theories only. 

2. $T(H,M)$ can be naturally described in terms of the {virtual $t$-truncations} of $H$
(starting from $E_2$); see \cite{bws}*{\S2.5} and \cite{brelmot}.

3. For $S$ being   a pro-smooth affine limit of varieties (cf. \S\ref{stmain})
we obtain that the (co)homology of any $M\in \obj\dmcs$ possesses a filtration by subfactors of (co)homology of regular projective $S$-schemes.

4. The arguments in \cite{cdet}*{\S7.2} easily yield (for $R\subset \Q$, $p\neq l\in \p$) the existence of a (covariant) \'etale realization functor from $\dmcs$ into the corresponding  category of Ekedahl's  (\'etale, constructible) ${\mathbb{Q}}_l$-adic  systems. 

Note next that for $S$ being a variety over $k$ and $p>0$   the target category is endowed with a certain weight filtration (that was defined in \cite{bbd}*{\S5}); for $p=0$ one can factor the realization functor mentioned through a certain category that is endowed with certain "weights" also (and was defined by A. Huber; see \cite{bmm}*{Proposition 2.5.1}). Moreover (see loc. cit.) in both cases the corresponding functors "respect weights". So, we obtain that $w_{\chow}$ is closely related to Deligne's weights for constructible complexes of sheaves (that were crucial for defining "weights" in \cite{huper} also). Yet (as explained in \cite{bmm}*{Remark 2.5.2}) the weight filtrations on the "\'etale categories" mentioned do not yield weight structures for them.  

\end{rema}

The functoriality of Chow-weight filtrations has quite interesting consequences.

\begin{proposition}\label{pscholl}
For a scheme $X$ let $H:\dmcx\to \underline{A}$ be a contravariant functor. 
For all $N\in  \dmc(U)_{w_{\chow} \le 0}$ consider $G(N)= (W^0H)(j_!(N))\subset H(j_!(N))$.
\begin{enumerate}
\item The following statements are fulfilled.
\begin{enumerate}
\item $G(N)$ is $\dmc(U)$-functorial in $N$.

\item 
 $G(N)$ is a quotient of $H(M)$ for some $M\in \dmcx_{w_{\chow}=0}$.
\end{enumerate}
\item Let now $N=j^*(M)(=j^!(M))$ for $M\in \dmcx_{w_{\chow}=0}$. Then also the following is true.
\begin{enumerate}
\item 
$G(N)=\imm H(M)\to H(j_!j^!(M))$  {\nf(}here we apply $H$ to the morphism $j_!j^!(M)\to M$ coming from the adjunction $j_!\dashv 
j^!=j^*${\nf)}.
\item Let $H=\dmcx(-,\mathbb{H})$ for some $\mathbb{H}\in \obj \dmx$. 
Then  $G(N)\cong\break \imm(\dmx(M, \mathbb{H})\to \dm(U)(N, j^*(\mathbb{H})))$. \end{enumerate}
\end{enumerate}\end{proposition}

\begin{rema}
1.  Thus one may say that  $(W^0H)(j_!(N))$ yields the 'integral part' of  $H(j_!(N))$: we obtain the subobject of  $H^*(j_!(N))$ that  'comes from any nice $X$-lift' of $N$ if the latter exists, and a factor of $H(M)$ for some $M\in \dmcx_{w_{\chow}=0}$ in general;
cf. \cite{bei85} and \cite{scholl}. If $N=\mgbm(T_U)$ for a regular $T_U$ 
proper over $U$, then any regular proper "$X$-model" $T_X$ for $T_U$ yields such a "nice $X$-lift" for $N$.
Besides, in the setting of assertion 2b one can take $\mathbb{H}$ representing \'etale or motivic cohomology (for an appropriate choice of coefficients and some fixed degree; so that $\dms(R_Y,f^*(\mathbb{H}))$ is the corresponding cohomology of $Y$ for any $f:Y\to X$); then $G(N)$ will be the image of this cohomology of $T_X$ in the one of $T_U$.

In this case one can also 
study the cohomology of an object $N_K$ of $ \chow(K)$ (i.e., of an element of $\dmc(K)_{w_{\chow}=0}$) for $K$ being some generic point of $X$. Indeed,  any such $N_K$ can be lifted to a Chow motif $N$ over some $U$ ($K\in U$, $U$ is open in $X$) by 
Remark \ref{ridmot}(2)   combined with 
Theorem \ref{pcisdeg}(\ref{icont}) (since if $N_K$ is a retract of $\mgbm_K(T_K)$ for some regular variety $T_K/K$, then $T_K$ along with this retraction  can be lifted  to a regular $T_U$ over some open $U\subset X$ that contains $K$). 
Note that this construction enjoys 'the usual' $\chow(K)$-functoriality; this is an easy consequence of  Theorem \ref{pcisdeg}(\ref{icont}).

2. It could also be interesting to consider $W^lH^*(j_!(N))$ for $l< 0$.
\end{rema}

\end{document}